\documentclass{article}

\usepackage[english]{babel}

\usepackage[letterpaper,top=2cm,bottom=2cm,left=3cm,right=3cm,marginparwidth=1.75cm]{geometry}

\usepackage{amsmath}
\usepackage{amssymb}
\usepackage{amsthm}
\usepackage{graphicx}
\usepackage[colorlinks=true, allcolors=blue]{hyperref}
\usepackage{xfrac}
\usepackage{tikz}
\usetikzlibrary{arrows.meta}
\usetikzlibrary{math}
\usepackage{array}
\usepackage{pgfplots}
\pgfplotsset{compat=1.18}

\allowdisplaybreaks

\theoremstyle{plain}
\newtheorem{theorem}{Theorem}
\newtheorem{lemma}[theorem]{Lemma}
\newtheorem{proposition}[theorem]{Proposition}

\theoremstyle{definition}

\theoremstyle{remark}
\newtheorem*{remark}{Remark}

\newcommand{\Andras}{Andr\'as}
\newcommand{\Erdos}{Erd\H{o}s}
\newcommand{\Sarkozy}{S\'ark\"ozy}
\newcommand{\Lovasz}{Lov\'asz}
\newcommand{\Chvatal}{Chv\'atal}

\title{The independence and clique cover numbers\\of the squarefree graph}
\author{Boris Alexeev \and Dustin G.\ Mixon\thanks{Department of Mathematics, The Ohio State University, Columbus, OH} \thanks{Translational Data Analytics Institute, The Ohio State University, Columbus, OH} \and Will Sawin\thanks{Department of Mathematics, Princeton University, Princeton, NJ}}
\date{}

\begin{document}
\maketitle

\begin{abstract}
We determine the largest subset $A\subseteq \{1,\dotsc,n\}$ such that for all $a,b\in A$, the product $ab$ is not squarefree.
Specifically, the maximum size is achieved by the \emph{complement} of the odd squarefree numbers.

This resolves a problem of Paul \Erdos{} and \Andras{} \Sarkozy{} from 1992.
\end{abstract}

\section{Introduction}

In 1992, \Erdos{} wrote~\cite{Er92b}:
\begin{quote}
\textbf{23.}
Here is a recent problem of \Sarkozy{} and myself: Let $a_1<a_2<\dotsb<a_k\le n$ be a sequence of [positive] integers.
Assume that the product $a_i a_j$ is never squarefree.
When is $k$ maximal?
Our obvious guess was: $k$ is maximal if the $a$'s are the even numbers and the odd non-squarefree numbers.
\end{quote}
We will prove that this guess is correct.

The online collection \url{erdosproblems.com} includes this problem~\cite{dotcom} as~\#844.
The authors thank Thomas Bloom for curating this very useful resource.
(We note that just before the completion of this paper, that website was updated to include a solution to this problem.
We will discuss this in detail shortly.)

In a different language, one may define a graph with vertex set $\{1,\ldots,n\}$ and edges connecting $a$ and $b$ if $ab$ is squarefree.
Then \Erdos{} and \Sarkozy{} ask for the independence number of this graph, that is, the size of a maximum independent set (a set of vertices with no edges connecting them, also known as a stable set or co/anti-clique).\footnote{For $n=1$, there arises a technical issue regarding whether $\{1\}$ is an independent set.  This issue may be resolved either by considering $1$ to have a self-loop, which disqualifies $1$ from being in an independent set, or alternatively, by only considering $n\geq2$, where this issue does not change the answer.}

If $a$ is \emph{not} squarefree, then it is isolated (not adjacent to any other vertex) and so may be freely included in any independent set.
Accordingly, it makes sense to focus our attention only to squarefree vertices.
Let the \emph{squarefree graph} up to $n$ have as vertices only the squarefree numbers in $\{1,\ldots,n\}$, and continue to let vertices $a,b$ be adjacent if $ab$ is squarefree.
Because $a$ and $b$ are themselves squarefree, this condition is now equivalent to $a$ and $b$ being coprime.
The guess of \Erdos{} and \Sarkozy{} now corresponds to the following statement.

\begin{theorem}
\label{thm.main}
For the squarefree graph, the even vertices form a maximum independent set.
\end{theorem}

For example, the case where $n=15$ is illustrated in Figure~\ref{fig:ex15}.

\begin{figure}
\centering
\begin{tikzpicture}[scale=3,thick]

  \def\nodelist{1,2,3,5,6,7,10,11,13,14,15}
  \def\numnodes{11}

  \foreach [count=\i] \k in \nodelist {
    \pgfmathsetmacro{\angle}{360/\numnodes*(\i - 1)}
    \ifnum\k=2 \node[draw,circle,fill=lightgray,minimum size=7mm,inner sep=0pt] (n\k) at ({cos(\angle)}, {sin(\angle)}) {\k};
    \else\ifnum\k=6 \node[draw,circle,fill=lightgray,minimum size=7mm,inner sep=0pt] (n\k) at ({cos(\angle)}, {sin(\angle)}) {\k};
    \else\ifnum\k=10 \node[draw,circle,fill=lightgray,minimum size=7mm,inner sep=0pt] (n\k) at ({cos(\angle)}, {sin(\angle)}) {\k};
    \else\ifnum\k=14 \node[draw,circle,fill=lightgray,minimum size=7mm,inner sep=0pt] (n\k) at ({cos(\angle)}, {sin(\angle)}) {\k};
    \else \node[draw,circle,fill=white,minimum size=7mm,inner sep=0pt] (n\k) at ({cos(\angle)}, {sin(\angle)}) {\k};
    \fi\fi\fi\fi
  }

  \foreach \a in \nodelist {
    \foreach \b in \nodelist {
      \ifnum\a<\b
        \pgfmathparse{gcd(\a,\b)==1 ? 1 : 0}
        \ifnum\pgfmathresult=1
          \draw (n\a) -- (n\b);
        \fi
      \fi
    }
  }

\end{tikzpicture}
\qquad\quad
\begin{tikzpicture}[scale=3,thick]

  \def\nodelist{1,2,3,5,6,7,10,11,13,14,15}
  \def\numnodes{11}

  \foreach [count=\i] \k in \nodelist {
  \pgfmathsetmacro{\angle}{360/\numnodes*(\i - 1)}
  \ifnum\k=1 \node[draw,circle,fill=blue!40,minimum size=7mm,inner sep=0pt] (n\k) at ({cos(\angle)}, {sin(\angle)}) {\k};
  \else\ifnum\k=2 \node[draw,circle,fill=blue!40,minimum size=7mm,inner sep=0pt] (n\k) at ({cos(\angle)}, {sin(\angle)}) {\k};
  \else\ifnum\k=3 \node[draw,circle,fill=blue!40,minimum size=7mm,inner sep=0pt] (n\k) at ({cos(\angle)}, {sin(\angle)}) {\k};
  \else\ifnum\k=5 \node[draw,circle,fill=blue!40,minimum size=7mm,inner sep=0pt] (n\k) at ({cos(\angle)}, {sin(\angle)}) {\k};

  \else\ifnum\k=6 \node[draw,circle,fill=green!40,minimum size=7mm,inner sep=0pt] (n\k) at ({cos(\angle)}, {sin(\angle)}) {\k};
  \else\ifnum\k=7 \node[draw,circle,fill=green!40,minimum size=7mm,inner sep=0pt] (n\k) at ({cos(\angle)}, {sin(\angle)}) {\k};
  \else\ifnum\k=11 \node[draw,circle,fill=green!40,minimum size=7mm,inner sep=0pt] (n\k) at ({cos(\angle)}, {sin(\angle)}) {\k};
  \else\ifnum\k=13 \node[draw,circle,fill=green!40,minimum size=7mm,inner sep=0pt] (n\k) at ({cos(\angle)}, {sin(\angle)}) {\k};

  \else\ifnum\k=10 \node[draw,circle,fill=yellow!40,minimum size=7mm,inner sep=0pt] (n\k) at ({cos(\angle)}, {sin(\angle)}) {\k};

  \else\ifnum\k=14 \node[draw,circle,fill=red!40,minimum size=7mm,inner sep=0pt] (n\k) at ({cos(\angle)}, {sin(\angle)}) {\k};
  \else\ifnum\k=15 \node[draw,circle,fill=red!40,minimum size=7mm,inner sep=0pt] (n\k) at ({cos(\angle)}, {sin(\angle)}) {\k};

  \else \node[draw,circle,fill=white,minimum size=7mm,inner sep=0pt] (n\k) at ({cos(\angle)}, {sin(\angle)}) {\k};
  \fi\fi\fi\fi\fi\fi\fi\fi\fi\fi\fi
  }
  

  \foreach \a in \nodelist {
    \foreach \b in \nodelist {
      \ifnum\a<\b
        \pgfmathparse{gcd(\a,\b)==1 ? 0 : 1}
        \ifnum\pgfmathresult=1
          \draw (n\a) -- (n\b);
        \fi
      \fi
    }
  }

\end{tikzpicture}
\caption{\textbf{(left)} The squarefree graph for $n=15$. The vertex set consists of all squarefree integers in $\{1,\ldots,n\}$, and vertices are adjacent precisely when their product is squarefree (or equivalently, when they are coprime). Consistent with the guess of \Erdos{} and \Sarkozy{}, the even vertices form the largest possible independent set of this graph. \textbf{(right)} The complement of the squarefree graph.  Vertices now share an edge when they share a factor, and it is more visibly discernible that the even vertices form a maximum clique. We color the vertices to illustrate how the chromatic number equals the clique number. In this paper, we prove that this occurs for all $n$.}
\label{fig:ex15}
\end{figure}
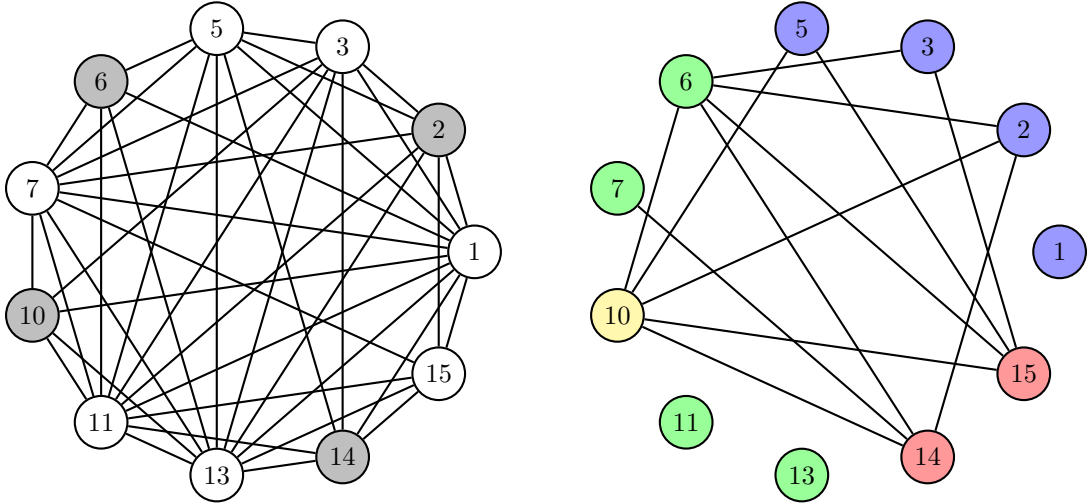

\begin{remark}
As we shall discuss in greater detail later, it is known that asymptotically the proportion of numbers that are squarefree is $\frac{6}{\pi^2} \approx 60.8\%$.
Moreover, among squarefree integers, asymptotically one-third are even and two-thirds are odd; thus among all numbers, the proportion that are even and squarefree is $\frac{2}{\pi^2} \approx 20.3\%$, and the proportion that are odd and squarefree is $\frac{4}{\pi^2} \approx 40.5\%$.
In terms of the squarefree graph, since the even vertices form a maximum independent set, the independence number is asymptotically one-third of the number of vertices.
In terms of the entire set $\{1,\ldots,n\}$, the complement of the odd squarefree numbers is a maximum independent set, so the independence number is asymptotically $1-\frac{4}{\pi^2}\approx 59.5\%$ of $n$.
\end{remark}

During the preparation of this paper, Desmond Weisenberg~\cite{dotcom} observed that Theorem~\ref{thm.main} is an easy consequence of a result of \Chvatal{}~\cite{chvatal}.
We describe this proof in Subsection~\ref{subsec.weisenberg}.
Perhaps surprisingly, \Chvatal{}'s result predates the original problem statement by two decades, and \Erdos{} was even aware of it!
Still, as we will see, \Erdos{}’s problem is vindicated to some extent by the curious behavior of several related invariants of the squarefree graph.
In particular, we present an entirely different proof of Theorem~\ref{thm.main} by proving an even stronger result:

\begin{theorem}
\label{thm.stronger}
The vertices of the squarefree graph can be partitioned into cliques, each containing exactly one even vertex.
\end{theorem}

Notably, since the even vertices form an independent set, Theorem~\ref{thm.stronger} implies that the clique cover number (i.e., the chromatic number of the complement), the independence number, and even the \Lovasz{} number of the squarefree graph all equal the number of even squarefree numbers up to $n$.

(For the full graph on $\{1,\dotsc,n\}$, the analogous equalities hold as well, because as mentioned earlier, the non-squarefree numbers are all isolated.)

\subsection{Weisenberg's proof of Theorem~\ref{thm.main}}
\label{subsec.weisenberg}

Weisenberg~\cite{dotcom} recalls the following result of \Chvatal{}~\cite{chvatal}:

\begin{proposition}
\label{prop.chvatal}
Let $\mathcal{F}$ be a family of subsets of $\{1,\dotsc,k\}$ such that whenever $A\in\mathcal{F}$ and there is an injection $f\colon B\to A$ such that $x\le f(x)$ for all $x\in B\subseteq\{1,\dotsc,k\}$, then $B\in\mathcal{F}$.
Then whenever $\mathcal{F'}\subseteq\mathcal{F}$ is an intersecting subfamily, we have
\[
\# \mathcal{F'} \le \#\{A\in\mathcal{F}: 1\in A\}.
\]
(An intersecting subfamily $\mathcal{F'}$ is one where for any two subsets $A,B\in\mathcal{F'}$, we have $A\cap B\ne\emptyset$.)
\end{proposition}

To apply this to the squarefree graph, let $\mathcal{F}$ be the family of prime factors of the vertices, with the $i$th prime $p_i$ encoded simply as $i$.
That is, a subset $A\subseteq\{1,\dotsc,\pi(n)\}$ is a member of $\mathcal{F}$ if $\prod_{i\in A} p_i\le n$.
An intersecting subfamily $\mathcal{F'}$ of $\mathcal{F}$ corresponds to an independent set in the squarefree graph.
Thus, Proposition~\ref{prop.chvatal} implies that a maximum independent set is given by all multiples of the prime $p_1=2$.

\subsection{Outline}

In the next section, we start tackling Theorem~\ref{thm.stronger} by making a combinatorial reduction of sorts.
Here, we present multiple strategies for constructing the desired clique cover.
We verify computationally that one of these strategies works for small $n$, while for large $n$, we analyze another strategy via the conditions of Lemma~\ref{lem.main}.
Section~\ref{sec.number theoretic} presents a number-theoretic reduction in order to prove Lemma~\ref{lem.main}.
Specifically, our proof uses the negative moment method, assuming the estimates in Lemma~\ref{lem.estimates}.
We then prove these estimates in Section~\ref{sec.estimates} and conclude with a brief discussion in Section~\ref{sec.discussion}.

\section{Combinatorial reduction}
\label{sec.combinatorial}

First, observe that if $a$ and $b$ are both even, then $ab$ is divisible by $4$ and thus not squarefree.
Thus, the even squarefree numbers constitute an independent set in the squarefree graph.

We will show that they form an independent set of maximum size by constructing a vertex clique cover of the same size.
That is, we will partition the vertices of our graph into parts, each of which induces a clique and each of which contains exactly one even vertex.
The \Lovasz{} sandwich theorem says that for any graph, the independence number is less than or equal to the \Lovasz{} number, which in turn is less than or equal to the clique cover number (the chromatic number of the complement).
Thus our results also determine the \Lovasz{} number of the squarefree graph.

We think of the cliques in our cover as being named by the even vertex they contain.
In other words, we construct our clique cover as the fibers (inverse images) of a map from the entire vertex set to the subset of even vertices.
Since each even vertex is in the clique with its own name, this map will assign each even vertex to itself, and it remains to determine how to map the odd vertices.
In this section, we present multiple different strategies to accomplish this and then use two of them in different settings to prove Theorem~\ref{thm.stronger} (and thus Theorem~\ref{thm.main}).

\subsection{A greedy strategy}

A straightforward approach to assigning the odd vertices is \emph{greedy}:
We iteratively assign each odd vertex $\ell$, in increasing order, to the smallest even vertex $m$ where it is still allowed.
In order for an assignment to be allowed, $\ell$ must be coprime to $m$ and all odd numbers already assigned to $m$, thereby preserving the property that all vertices assigned to $m$ induce a clique.
For example, after performing this strategy for $n=100$, the partition obtained looks as follows:
\begin{center}
\setlength{\tabcolsep}{3pt}
\begin{tabular}{r|rrrrrrrrrrrrrrrrrrrrrrrrr}
2 & 2 & 3 & 5 & 7 & 11 & 13 & 17 & 19 & 23 & 29 & 31 & 37 & 41 & 43 & 47 & 53 & 59 & 61 & 67 & 71 & 73 & 79 & 83 & 89 & 97\\
6 & 6 & 35\\
10 & 10 & 21\\
14 & 14 & 15\\
22 & 22 & 39 & 85\\
26 & 26 & 33 & 95\\
30 & 30 & 77\\
34 & 34 & 55 & 57 & 91\\
38 & 38 & 51 & 65\\
46 & 46 & 87\\
58 & 58 & 69\\
70 & 70 & 93\\
\end{tabular}
\end{center}
The vertices assigned to the even vertex $2$ are all of the prime numbers.
The vertices assigned to the even vertex $6$ are products of pairs of consecutive primes: $2\times 3, 5\times 7, 11\times 13, 17\times 19, 23\times 29, \dotsc$ (but not $3\times 5$, as that would not be relatively prime to $2\times 3$).
The vertices assigned to larger even vertices are somewhat harder to describe.

We were able to run this strategy successfully for $n\le 1.8\cdot 10^8$ (one hundred eighty million).
We further believe that this strategy succeeds for all $n$, but we do not have a proof.
Meanwhile, computationally, we found that a small modification to the strategy was much faster to run.

\subsection{A faster greedy strategy}
\label{subsec.greedy}

The greedy strategy of the previous subsection continues to assign odd vertices to the same even vertex indefinitely.
To implement this strategy, one must track prime factors for a growing number of cliques.
To alleviate this memory bottleneck, we can modify the strategy by calling an even vertex ``done'' once it has an appropriate number of vertices assigned to it.

Specifically, for our faster greedy strategy, we iteratively assign each odd vertex $\ell$ to the smallest even vertex coprime to $\ell$ which is allowed and furthermore currently has fewer than \emph{three} assigned odd vertices.
This strategy is extremely efficient from a computational perspective, since we only need to keep a small subset of ``unfilled'' even vertices (and their currently assigned odd vertices) in working memory.
(The number three was chosen because two doesn't work and four ends up slower.)
This strategy will deliver the desired partition, provided such an even vertex always exists.
Figure~\ref{fig:greedy} illustrates this strategy in the case of $n=15$.

We believe that this strategy also succeeds for all $n$, but we again do not have a proof.  However, using this strategy, we were able to computationally verify Theorem~\ref{thm.stronger} for ``small'' $n$, specifically all $n\leq 3.5\cdot 10^{10}$ (thirty-five billion).

\begin{figure}[t]
\centering
\begin{tikzpicture}[>=Latex, thick]
  \foreach \k in {1,2,3,5,6,7,10,11,13,14,15} {
    \node[draw,circle,minimum size=7mm,inner sep=0pt] (n\k) at (\k,3) {\k};
  }
  \foreach \k in {2,6,10,14} {
    \node[draw,circle,minimum size=7mm,inner sep=0pt] (m\k) at (\k,0) {\k};
  }
  \draw[->] (n1)  -- (m2);
  \draw[->] (n2)  -- (m2);
  \draw[->] (n3)  -- (m2);
  \draw[->] (n5)  -- (m2);
  \draw[->] (n6)  -- (m6);
  \draw[->] (n7)  -- (m6);
  \draw[->] (n10) -- (m10);
  \draw[->] (n11) -- (m6);
  \draw[->] (n13) -- (m6);
  \draw[->] (n14) -- (m14);
  \draw[->] (n15) -- (m14);
\end{tikzpicture}
\caption{The greedy strategy described in Subsection~\ref{subsec.greedy}.
The fibers of this map are the cliques $\{1,2,3,5\}$, $\{6,7,11,13\}$, $\{10\}$, and $\{14,15\}$ in the squarefree graph with $n=15$. These are the color classes illustrated in Figure~\ref{fig:ex15} (right).
While the assignment $15\mapsto14$ seems like a ``close call'' (in the sense that $14$ is ``almost'' larger than $15$), this appears to be the last time this strategy ever assigns an odd number to its even predecessor.}
\label{fig:greedy}
\end{figure}
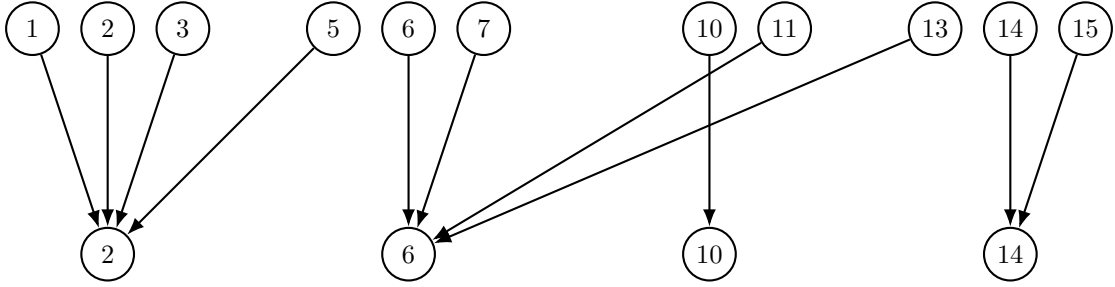

\subsection{A most-constrained-first strategy}
\label{subsec.most-constrained}

Next, we describe an alternative strategy that is more amenable to analysis.

Consider any strategy that iteratively assigns odd vertices (in some order) to even vertices.
Each step of the assignment will work provided there is a ``free'' even vertex.
Specifically, an odd vertex $\ell$ may be assigned to an even vertex $m$ as long as $\ell$ and $m$ are coprime, and no other odd vertex $\ell'$ that shares a factor with $\ell$ has already been assigned to $m$.
Note that by the pigeonhole principle, such an even vertex will necessarily exist if the number of even vertices $m$ that are coprime to $\ell$ is greater than the number of odd vertices $\ell'$ that (a) have already been assigned and (b) share a factor with $\ell$.
Unfortunately, depending on the order of assignment, it is possible for this condition to be violated.

A conservative order of assignment is given by a \textit{most-constrained-first} strategy.
Here, we assign odd vertices in increasing order of number of coprime even vertices.
For example, when $n=15$, this determines the following ordering of odd vertices:
\begin{center}
\renewcommand{\arraystretch}{1.3}
\begin{tabular}{|c|*{4}{>{\centering\arraybackslash}p{2em}}|c|}\hline
~$\ell$~ & \multicolumn{4}{c|}{even vertices coprime to $\ell$} & assignment \\ \hline
~15~  & 2   & 14  &     &     & 2 \\
 ~3~  & 2   & 10  & 14  &     & 10 \\
 ~5~  & 2   & 6   & 14  &     & 6 \\
 ~7~  & 2   & 6   & 10  &     & 2 \\
 ~1~  & 2   & 6   & 10  & 14  & 2 \\
~11~  & 2   & 6   & 10  & 14  & 2 \\
~13~  & 2   & 6   & 10  & 14  & 2 \\\hline
\end{tabular}
\end{center}
(When two vertices are coprime to the same number of even vertices, we break ties by listing them in vertex order.)
Since $15$ is the odd vertex that is coprime to the fewest even vertices, we assign it first.
When there are multiple choices available, we always use the least available assignment (at least in this example), so we take $15\mapsto2$.
Next, we assign~$3$.
Since $15$ shares a factor with~$3$, it blocks the would-be assignment~$2$, and so the least available assignment is $3\mapsto10$.
Toward the end of this process, we only have to assign~$1$ and the primes greater than~$\frac{n}{2}$, each of which is coprime to all other vertices, and so their assignments are completely unconstrained.

In order to prove that this strategy succeeds, we simply need to show that at all steps, there is a ``free'' even vertex:
For every odd vertex $\ell$, the number of even vertices coprime to $\ell$ is greater than or equal to the number of odd vertices $\ell'$ that both (a) are coprime to as few (or fewer) even vertices as $\ell$ and (b) share a factor with $\ell$.
(We now say ``greater than or equal to'' because, as written, $\ell'=\ell$ is included in the latter count.)
Analyzing the number of odd vertices $\ell'$ that satisfy both (a) and (b) is a bit tricky, but it turns out that depending on $\ell$, we can drop one of these constraints and still show that the strategy succeeds.
Specifically, if there are only a few even vertices coprime to $\ell$, we will keep condition (a) but ignore condition (b), while if there many even vertices coprime to $\ell$, we will keep condition (b) but ignore condition (a).

In other words, when $n$ is large, this strategy succeeds thanks to an overall dearth of constraints, as made explicit by the following lemma:

\begin{lemma}
\label{lem.main}
For every $n\geq10^{10}$, there exists $K=K(n)$ such that the following hold:
\begin{itemize}
\item[(a)]
For each $k\leq K$, there are at most $k$ odd vertices that are each coprime to at most $k$ even vertices.
\item[(b)]
Each odd vertex coprime to $k>K$ even vertices shares a factor with at most $k$ odd vertices.
\end{itemize}
\end{lemma}

For example, despite the hypothesis $n\geq10^{10}$, we may take $K(15)=2$: there are zero odd vertices coprime to at most one even vertex, there is one odd vertex (namely, $15$) that is coprime to at most two even vertices, and every odd vertex coprime to more than two even vertices is either $1$ or prime, so they share a factor with at most two odd vertices (namely, themselves and possibly $15$).

In the next subsection, we use Lemma~\ref{lem.main} to prove Theorem~\ref{thm.stronger} for large $n$.
Our proof of Lemma~\ref{lem.main}, which occupies the next two sections, actually works for $n\geq 6\cdot 10^9$ (six billion).
As with our other two strategies, we believe that Lemma~\ref{lem.main} holds for all $n$, but our estimates don't kick in until $n$ is sufficiently large.
We have computationally verified that Lemma~\ref{lem.main} also holds for $n\leq2\cdot 10^4$ (twenty thousand), but we are missing a range of ``medium''-sized $n$.
(In practice, our implementation of this strategy exhibits much longer runtimes than the faster greedy strategy of the previous subsection, though it can certainly be improved, perhaps even to capture the remaining range of $n$.)

\subsection{Proof of Theorem~\ref{thm.stronger}}

We proceed in cases.

\medskip

\noindent
\textbf{Case I:} $n<10^{10}$.
We ran the faster greedy strategy to find the desired partition for all such $n$.

\medskip

\noindent
\textbf{Case II:} $n\geq10^{10}$.
We claim that the most-constrained-first strategy finds the desired partition.
First, for the odd vertices that are each coprime to at most $K$ even vertices, Lemma~\ref{lem.main}(a) allows our strategy to assign them to distinct even vertices.
Next, Lemma~\ref{lem.main}(b) ensures that our strategy is not obstructed from assigning the remaining odd vertices.

\section{Number-theoretic reduction}
\label{sec.number theoretic}

Our proof of Lemma~\ref{lem.main} makes use of nonasymptotic versions of certain asymptotic behaviors of the squarefree graph.
Specifically, the size of the vertex set $V(n)$ is asymptotically
\[
v(n):=\frac{6}{\pi^2}\cdot n.
\]
Of these vertices, $\sim\frac{2}{3}$ are odd.
That is, the vertex~$2$ has degree $\sim\frac{2}{3}v(n)$.
More generally, denoting
\[
f(\ell)
:=\prod_{p\mid\ell}\bigg(1-\frac{1}{p+1}\bigg),
\]
then the vertex $\ell$ has degree $\sim f(\ell)\cdot v(n)$.
If $\ell$ is odd, then $\sim\frac{1}{3}$ of these neighbors are even.
Similarly, $\sim\frac{2}{3}$ of its $\sim(1-f(\ell))\cdot v(n)$ non-neighbors are odd.
In particular, every odd $\ell$ has more even neighbors than odd non-neighbors when $f(\ell)>\frac{2}{3}$, at least asymptotically.
This suggests that Lemma~\ref{lem.main}(b) should hold for all sufficiently large $n$ provided we take $K\geq(\frac{2}{3}+\varepsilon) \cdot\frac{1}{3}v(n)$.
Next, an asymptotic version of Lemma~\ref{lem.main}(a) states that if $\ell$ is drawn uniformly at random from the odd vertices, then 
\[
\mathbb{P}\bigg\{\frac{1}{3}f(\ell)v(n)\leq k\bigg\}
\leq\frac{k}{\frac{2}{3}v(n)}
\qquad
\forall\,k\leq K.
\]
Changing variables $k=t\cdot \frac{1}{3}v(n)$, this simplifies to
\begin{equation}
\label{eq.tail bound}
\mathbb{P}\{f(\ell)\leq t\}
\leq\frac{t}{2}
\qquad
\forall\,t\leq\frac{K}{\frac{1}{3}v(n)}.
\end{equation}
Considering Figure~\ref{fig:cdf}, this appears to hold whenever $t\leq0.874$, suggesting that one may take $K$ to be as large as $0.874\cdot \frac{1}{3}v(n)$.
For simplicity, our analysis will not use the full detail of the distribution of (the nonasymptotic analog of) $f(\ell)$, so we instead obtain a smaller choice of $K$ that works.

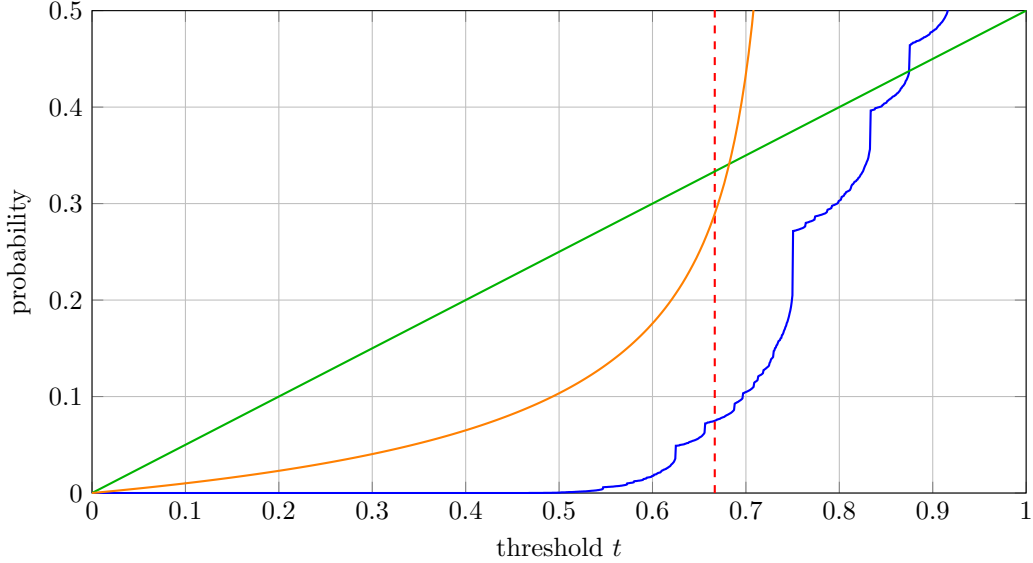
\begin{figure}[t]
\centering
\begin{tikzpicture}
  \begin{axis}[
    xlabel={threshold $t$},
    ylabel={probability},
    xmin=0, xmax=1,
    ymin=0, ymax=0.5,
    grid=major,
    width=14cm,
    height=8cm
  ]
    \addplot[
      color=blue,
      thick,
    ] 
    table[
      col sep=comma
    ] {data.txt};

    \addplot[
      color=red,
      thick,
      dashed
    ] coordinates {(2/3, 0) (2/3, 1)};
    
    \addplot[
      domain=0:1,
      samples=100,
      color=green!70!black,
      thick
    ] {x/2};

    \pgfmathsetmacro{\A}{16*3.1416^2/(91*1.202057)} 
    \pgfmathsetmacro{\B}{12*3.1416^2/(91*1.202057)}

    \addplot[
     thick,
     orange,
     domain=0:0.74,
     samples=200
    ]
    {
     (\A - 4/3) / (1/x - 4/3) * (1/4)
     + (\B - 1)/(1/x - 1) * (3/4)
    };
  \end{axis}
\end{tikzpicture}
\caption{In blue, we graph the asymptotic cumulative distribution function of $f(\ell)$ for $\ell$ drawn uniformly from the odd vertices. Comparing with the green line, we have $\mathbb{P}\{f(\ell)\leq t\}
\leq t/2$ for $t\leq0.874$. To prove an asymptotic version of Lemma~\ref{lem.main}, it suffices to establish this inequality for $t\leq\frac{2}{3}+\varepsilon$. By conditioning on whether $3\mid\ell$, the negative moment method delivers the upper bound in orange.}
\label{fig:cdf}
\end{figure}

For each odd $\ell\in V(n)$, let $C(\ell,n)$ denote the number of even vertices coprime to $\ell$, and put
\[
c(\ell,n)
:=\frac{1}{3}f(\ell)v(n),
\qquad
F(\ell,n):=\frac{C(\ell,n)}{c(1,n)},
\]
so that 
\[
C(\ell,n)\sim c(\ell,n),
\qquad
F(\ell,n)\sim f(\ell).
\]
(As a mnemonic, we generally use lowercase letters to denote the asymptotic version of quantities denoted by the corresponding capital letters.)
The nonasymptotic properties we use are captured by the following estimates, which we prove in the next section:

\begin{lemma}
\label{lem.estimates}
For every $n\geq 10^{10}$, the following hold:
\begin{itemize}
\item[(a)]
$\#\{\ell\in V(n):2\nmid\ell\}
=\big(1+E(\varepsilon_a)\big)\cdot\frac{2}{3}v(n)$ with $\varepsilon_a:=8.5\cdot10^{-5}$,
\item[(b)]
$\#\{\ell\in V(n):2\nmid\ell,\,3\nmid\ell\}
=\big(1+E(\varepsilon_b)\big)\cdot\frac{1}{2}v(n)$ with $\varepsilon_b:=1.8\cdot10^{-4}$,
\item[(c)]
$F(\ell,n)
=\big(1+E(\varepsilon_c)\big)\cdot f(\ell)$ for each odd $\ell\leq n$ with $\varepsilon_c:=3.8\cdot10^{-3}$,
\item[(d)]
$\sum_{\ell\in V(n),\,2\nmid\ell,\,3\nmid\ell}f(\ell)^{-1}
=\big(1+E(\varepsilon_d)\big)\cdot\frac{36}{91\zeta(3)}\cdot n$ with $\varepsilon_d:=1.3\cdot10^{-3}$,
\item[(e)]
$\sum_{\ell\in V(n),\,2\nmid\ell,\,3\mid\ell}f(\ell)^{-1}
=\big(1+E(\varepsilon_e)\big)\cdot\frac{16}{91\zeta(3)}\cdot n$ with $\varepsilon_e:=2.2\cdot10^{-3}$,
\end{itemize}
where $E(\varepsilon)$ denotes an error term whose absolute value is at most~$\varepsilon$.
Similar to the big~O notation $O(\varepsilon)$, the exact value of $E(\varepsilon)$ is possibly different with each use.
\end{lemma}

We will prove the relevant nonasymptotic analog of \eqref{eq.tail bound} using the \textit{negative moment method}:

\begin{lemma}
\label{lem.negative moment method}
Suppose $X\in[0,b]$ almost surely and $\mathbb{E}[X^{-1}] < \infty$.
Then for every $s\in(0,b)$,
\[
\mathbb{P}\{X\leq s\}
\leq\frac{\mathbb{E}[X^{-1}]-b^{-1}}{s^{-1}-b^{-1}}.
\]
\end{lemma}

\begin{proof}
Put $Y:=X^{-1}-b^{-1}$ and $t:=s^{-1}-b^{-1}$.
Then Markov's inequality gives
\[
\mathbb{P}\{X\leq s\}
=\mathbb{P}\{X^{-1}\geq s^{-1}\}
=\mathbb{P}\{Y\geq t\}
\leq\frac{\mathbb{E}Y}{t}
=\frac{\mathbb{E}[X^{-1}]-b^{-1}}{s^{-1}-b^{-1}}
\qedhere
\]
\end{proof}

\begin{proof}[Proof of Lemma~\ref{lem.main}]
We use Lemma~\ref{lem.estimates} to verify the result with
\[
K
:=0.672\cdot c(1,n).
\]
(Notably, $0.672$ is the notion of $\frac{2}{3}+\varepsilon$ that works well with our nonasymptotic analysis.)

For (a), draw $\ell$ uniformly from the odd vertices.
We wish to show that for every $k\leq K$,
\[
\mathbb{P}\{C(\ell,n)\leq k\}
\leq \frac{k}{\#\{\ell\in V(n):2\nmid\ell\}}.
\]
To accomplish this, we apply Lemma~\ref{lem.negative moment method} after conditioning on whether $3$ divides $\ell$.
(Without conditioning, the resulting bound is too weak for our purposes.)
Conditioning gives
\[
\mathbb{P}\{C(\ell,n)\leq k\}
=\mathbb{P}_{3\mid\ell}\bigg\{F(\ell,n)\leq\frac{k}{c(1,n)}\bigg\}\cdot\mathbb{P}\{\,3\mid\ell\,\}+\mathbb{P}_{3\nmid\ell}\bigg\{F(\ell,n)\leq\frac{k}{c(1,n)}\bigg\}\cdot\mathbb{P}\{\,3\nmid\ell\,\}.
\]
Lemma~\ref{lem.estimates}(a) and (b) together give
\[
\mathbb{P}\{\,3\nmid\ell\,\}
=\frac{\#\{\ell\in V(n):2\nmid\ell,\,3\nmid\ell\}}{\#\{\ell\in V(n):2\nmid\ell\}}
\leq\frac{(1+\varepsilon_b)\cdot\frac{1}{2}v(n)}{(1-\varepsilon_a)\cdot\frac{2}{3}v(n)}
=\frac{3(1+\varepsilon_b)}{4(1-\varepsilon_a)},
\]
\[
\mathbb{P}\{\,3\mid\ell\,\}
=1-\mathbb{P}\{\,3\nmid\ell\,\}
\leq1-\frac{3(1-\varepsilon_b)}{4(1+\varepsilon_a)}.
\]
Considering $f(\ell)\leq\frac{3}{4}$ whenever $2\nmid\ell$ and $3\mid\ell$, Lemma~\ref{lem.estimates}(c) gives that $F(\ell,n)\leq\frac{3}{4}\cdot(1+\varepsilon_c)$ for all such $\ell$.
Then Lemma~\ref{lem.negative moment method} implies
\[
\mathbb{P}\{C(\ell,n)\leq k\}
\leq\frac{\mathbb{E}_{3\mid\ell}[F(\ell,n)^{-1}]-(\frac{3}{4}(1+\varepsilon_c))^{-1}}{(\frac{k}{c(1,n)})^{-1}-(\frac{3}{4}(1+\varepsilon_c))^{-1}}\cdot\bigg(1-\frac{3(1-\varepsilon_b)}{4(1+\varepsilon_a)}\bigg)+\frac{\mathbb{E}_{3\nmid\ell}[F(\ell,n)^{-1}]-1}{(\frac{k}{c(1,n)})^{-1}-1}\cdot\frac{3(1+\varepsilon_b)}{4(1-\varepsilon_a)}.
\]
Next, we apply Lemma~\ref{lem.estimates}(c), and then (b) and (d) to get
\begin{align*}
\mathbb{E}_{3\nmid\ell}[F(\ell,n)^{-1}]
&\leq\frac{1}{1-\varepsilon_c}\cdot\mathbb{E}_{3\nmid\ell}[f(\ell)^{-1}]\\
&=\frac{1}{1-\varepsilon_c}\cdot\frac{1}{\#\{\ell\in V(n):2\nmid\ell,\,3\nmid\ell\}}\sum_{\substack{\ell\in V(n)\\2\nmid\ell,\,3\nmid\ell}}f(\ell)^{-1}
\leq\frac{1+\varepsilon_d}{(1-\varepsilon_c)(1-\varepsilon_b)}\cdot\frac{12\pi^2}{91\zeta(3)}.
\end{align*}
We similarly bound $\mathbb{E}_{3\mid\ell}[F(\ell,n)^{-1}]$ after applying Lemma~\ref{lem.estimates}(a) and (b) to estimate the underlying sample space:
\begin{align*}
\#\{\ell\in V(n):2\nmid\ell,\,3\mid\ell\}
&=\#\{\ell\in V(n):2\nmid\ell\}-\#\{\ell\in V(n):2\nmid\ell,\,3\nmid\ell\}\\
&\geq(1-\varepsilon_a)\cdot\frac{2}{3}v(n)-(1+\varepsilon_b)\cdot\frac{1}{2}v(n)\\
&=(1-4\varepsilon_a-3\varepsilon_b)\cdot\frac{1}{6}v(n).
\end{align*}
Combining this with Lemma~\ref{lem.estimates}(c) and (e) then gives
\[
\mathbb{E}_{3\mid\ell}[F(\ell,n)^{-1}]
\leq\frac{1}{1-\varepsilon_c}\cdot\frac{1}{\#\{\ell\in V(n):2\nmid\ell,\,3\mid\ell\}}\sum_{\substack{\ell\in V(n)\\2\nmid\ell,\,3\mid\ell}}f(\ell)^{-1}
\leq\frac{1+\varepsilon_e}{(1-\varepsilon_c)(1-4\varepsilon_a-3\varepsilon_b)}\cdot\frac{16\pi^2}{91\zeta(3)}.
\]
Denoting $t:=k/c(1,n)=3k/v(n)$, then
\begin{align*}
\mathbb{P}\{C(\ell,n)\leq k\}
&\leq\frac{\frac{1+\varepsilon_e}{(1-\varepsilon_c)(1-4\varepsilon_a-3\varepsilon_b)}\cdot\frac{16\pi^2}{91\zeta(3)}-(\frac{3}{4}(1+\varepsilon_c))^{-1}}{t^{-1}-(\frac{3}{4}(1+\varepsilon_c))^{-1}}\cdot\bigg(1-\frac{3(1-\varepsilon_b)}{4(1+\varepsilon_a)}\bigg)\\
&\qquad+\frac{\frac{1+\varepsilon_d}{(1-\varepsilon_c)(1-\varepsilon_b)}\cdot\frac{12\pi^2}{91\zeta(3)}-1}{t^{-1}-1}\cdot\frac{3(1+\varepsilon_b)}{4(1-\varepsilon_a)}.
\end{align*}
One may verify that for all $t\leq 0.672$, the right-hand side is less than $\frac{t}{2(1+\varepsilon_a)}$.
As such, for every $k\leq K:=0.672\cdot c(1,n)$, it holds that 
\[
\mathbb{P}\{C(\ell,n)\leq k\}
\leq \frac{t}{2(1+\varepsilon_a)}
= \frac{k}{(1+\varepsilon_a)\cdot\frac{2}{3}v(n)}
\leq\frac{k}{\#\{\ell\in V(n):2\nmid\ell\}},
\]
where the last step applies Lemma~\ref{lem.estimates}(a).

For (b), consider any odd vertex $\ell$ such that $C(\ell,n)> K:=0.672\cdot c(1,n)$. 
Note that by doubling the odd vertices coprime to $\ell$, we obtain the even squarefree numbers $\leq 2n$ that are coprime to $\ell$, of which there are $C(\ell,2n)$.
Also note that
\[
c(1,2n)
=\frac{1}{3}v(2n)
=\frac{2}{3}v(n)
=2c(1,n).
\]
It follows that the number of odd vertices that share a factor with $\ell$ is
\[
\begin{aligned}
&\#\{\ell'\in V(n):2\nmid\ell\}-C(\ell,2n)\\
&\qquad=\#\{\ell'\in V(n):2\nmid\ell\} - c(1,2n) F( \ell, 2n )&&\text{($F( \ell, 2n )=\tfrac{C(\ell,2n)}{c(1,2n)}$)} \\
&\qquad\leq(1+\varepsilon_a)\cdot\tfrac{2}{3}v(n) - c(1,2n)\cdot(1-\varepsilon_c)f(\ell)\qquad&&\text{(Lemma~\ref{lem.estimates}(a) and (c))}\\
&\qquad=2\big(1+\varepsilon_a-(1-\varepsilon_c)\cdot f(\ell)\big)\cdot c(1,n)&&\text{($c(1,n)=\tfrac{1}{3}v(n)$, $c(1,2n)=2c(1,n)$)}\\
&\qquad\leq2\big(1+\varepsilon_a-\tfrac{1-\varepsilon_c}{1+\varepsilon_c}\cdot F(\ell,n)\big)\cdot c(1,n)&&\text{(Lemma~\ref{lem.estimates}(c))}\\
&\qquad< 2\big(1+\varepsilon_a-\tfrac{1-\varepsilon_c}{1+\varepsilon_c}\cdot 0.672\big)\cdot c(1,n)&&\text{$(F( \ell, n )=\tfrac{C(\ell,n)}{c(1,n)}>0.672$)}\\
&\qquad\leq 0.672\cdot c(1,n)&&\text{($\varepsilon_a=8.5\cdot10^{-5}$, $\varepsilon_c=3.8\cdot10^{-3}$)}\\
&\qquad<C(\ell,n)&&\text{($C(\ell,n)>0.672\cdot c(1,n)$)},
\end{aligned}
\]
as desired.
\end{proof}

\section{Estimates}
\label{sec.estimates}

\subsection{Error bound on vertex degree estimate}

As a warmup, we first count the squarefree numbers in $\{1,\ldots,n\}$.
Given an integer $d$, let $A_d$ denote the set of numbers $\leq n$ divisible by $d$.
Then the squarefree numbers are given by
\[
V(n)
=\{1,\ldots,n\}\setminus\bigcup_p A_{p^2},
\]
where the union is taken over all primes $p$.
By inclusion--exclusion, we have
\begin{align*}
\#V(n)
&=n-\sum_p\#A_{p^2}+\sum_{p<q}\#(A_{p^2}\cap A_{q^2})-\sum_{p<q<r}\#(A_{p^2}\cap A_{q^2}\cap A_{r^2})+\cdots\\
&=n-\sum_p\bigg\lfloor\frac{n}{p^2}\bigg\rfloor+\sum_{p<q}\bigg\lfloor\frac{n}{p^2q^2}\bigg\rfloor-\sum_{p<q<r}\bigg\lfloor\frac{n}{p^2q^2r^2}\bigg\rfloor+\cdots\\
&=\sum_{d=1}^\infty\mu(d)\bigg\lfloor\frac{n}{d^2}\bigg\rfloor.
\end{align*}
Recalling $\sum_{d=1}^\infty\frac{\mu(d)}{d^2}
=\frac{6}{\pi^2}$, this suggests $\#V(n)\sim v(n)$.
In fact,
\begin{align*}
\big|\#V(n)-v(n)\big|
&=\bigg|\sum_{d=1}^\infty\mu(d)\bigg\lfloor\frac{n}{d^2}\bigg\rfloor-\sum_{d=1}^\infty\mu(d)\frac{n}{d^2}\bigg|\\
&\leq\sum_{d\leq\sqrt{n}+\frac{1}{2}}1+\sum_{d>\sqrt{n}+\frac{1}{2}}\frac{n}{d^2}
\leq\sqrt{n}+\frac{1}{2}+n\int_{\sqrt{n}}^\infty\frac{dx}{x^2}
=2\sqrt{n}+\frac{1}{2}.
\end{align*}
We mimic this argument to obtain the following result, which will help us prove Lemma~\ref{lem.estimates}.

\begin{lemma}
\label{lem.degree error}
For every vertex $\ell\in V(n)$, the degree of $\ell$ satisfies
\[
\big|\operatorname{deg}(\ell,n)-f(\ell)v(n)\big|
\leq2\sqrt{n}\cdot\prod_{p\mid\ell}\bigg(1+\frac{1}{\sqrt{p}}\bigg)+\frac{1}{2}\prod_{p\mid\ell}2,
\]
where the products are taken over primes $p$.
\end{lemma}
\begin{remark}
Later, we apply Lemma~\ref{lem.degree error} with a half-integer $n$.
By this, we mean one should take $\lfloor n\rfloor$ in $V(\lfloor n\rfloor)$, $\operatorname{deg}(\ell,\lfloor n\rfloor)$, and $\{1,\dotsc,\lfloor n\rfloor\}$, but not elsewhere.
\end{remark}
\begin{proof}
First, the neighborhood of $\ell$ is given by
\[
\{1,\ldots,n\}\setminus\bigg(\bigcup_{p\mid\ell}A_p\cup\bigcup_{p\nmid\ell}A_{p^2}\bigg),
\]
where the unions are taken over primes $p$.
An inclusion--exclusion argument then gives
\[
\operatorname{deg}(\ell,n)
=\sum_{a|\ell}\sum_{(b,\ell)=1}\mu(ab)\bigg\lfloor\frac{n}{ab^2}\bigg\rfloor.
\]
This suggests the approximation
\[
\operatorname{deg}(\ell,n)
\sim n\sum_{a|\ell}\sum_{(b,\ell)=1}\frac{\mu(ab)}{ab^2}
=f(\ell)v(n).
\]
In fact,
\begin{align*}
\big|\operatorname{deg}(\ell,n)-f(\ell)v(n)\big|
&=\bigg|\sum_{a|\ell}\sum_{(b,\ell)=1}\mu(ab)\bigg\lfloor\frac{n}{ab^2}\bigg\rfloor-\sum_{a|\ell}\sum_{(b,\ell)=1}\mu(ab)\frac{n}{ab^2}\bigg|\\
&\leq\sum_{a|\ell}\sum_{(b,\ell)=1}\bigg\{\frac{n}{ab^2}\bigg\}\\
&\leq\sum_{a|\ell}\bigg(\sum_{b\leq\sqrt{\frac{n}{a}}+\frac{1}{2}}1+\frac{n}{a}\sum_{b>\sqrt{\frac{n}{a}}+\frac{1}{2}}\frac{1}{b^2}\bigg)\\
&\leq\sum_{a|\ell}\bigg(2\sqrt{\frac{n}{a}}+\frac{1}{2}\bigg),
\end{align*}
where the last step applies the integral comparison from our warmup.
Finally,
\[
\sum_{a|\ell}\bigg(2\sqrt{\frac{n}{a}}+\frac{1}{2}\bigg)
=2\sqrt{n}\cdot\sum_{a|\ell}\frac{1}{\sqrt{a}}+\frac{1}{2}\sum_{a\mid\ell}1\\
=2\sqrt{n}\cdot\prod_{p\mid\ell}\bigg(1+\frac{1}{\sqrt{p}}\bigg)+\frac{1}{2}\prod_{p\mid\ell}2,
\]
which implies the claim.
\end{proof}

\subsection{Estimates on (sums of powers of) divisors}

The bound in Lemma~\ref{lem.degree error} motivates the following estimate:

\begin{lemma}
\label{lem.sigma}
Fix $\delta\geq0$.
For any odd squarefree $\ell$, it holds that
\[
\prod_{p\mid\ell}(1+p^{-\delta})
\leq \alpha_\delta \cdot \ell^{\beta_\delta}
\]
for constants
\[
\beta_{\delta} := \frac{\log(1+31^{-\delta})}{\log(31)},
\qquad
\alpha_{\delta}=3234846615^{-\beta_{\delta}} \prod_{2<p<31}(1+p^{-\delta}).
\]

\end{lemma}
\begin{proof}
Note that the function
\[
g_\delta(t,x)
:= t^{\frac{\log(1+x^{-\delta})}{\log x}}
\qquad
(t\geq0, ~ x>1)
\]
is multiplicative in $t$, decreasing in $x$, and satisfies $g_\delta(t,t)=1+t^{-\delta}$.
Thus,
\begin{align*}
\prod_{p\mid\ell}(1+p^{-\delta})
&=g_\delta(\ell,x)\cdot\prod_{p\mid\ell}\frac{1+p^{-\delta}}{g_\delta(p,x)}\\
&\leq g_\delta(\ell,x)\cdot\prod_{p\mid\ell}\frac{1+p^{-\delta}}{g_\delta(p,\max\{p,x\})}
=g_\delta(\ell,x)\cdot\prod_{\substack{p\mid\ell\\p<x}}\frac{1+p^{-\delta}}{g_\delta(p,x)}
\leq g_\delta(\ell,x)\cdot\prod_{\substack{2<p<x}}\frac{1+p^{-\delta}}{g_\delta(p,x)},
\end{align*}
where the restriction $p>2$ in the last step uses the fact that $\ell$ is odd.
In particular, for any fixed $x$, this bounds $\prod_{p\mid\ell}(1+p^{-\delta})$ by a power function of $\ell$ whose exponent is smaller when $x$ is larger.
Taking $x=31$ gives the bound as stated.
\end{proof}

Note that we are particularly interested in $\delta\in\{0,\frac{1}{2},1\}$, in which case
\begin{align*}
\alpha_0 &\le 6.1620, 
& \alpha_{1/2} &\le 3.9926, 
& \alpha_{1} &\le 2.1110,\\
\beta_0 &\le 0.2019, 
& \beta_{1/2} &\le 0.0482, 
& \beta_{1} &\le 0.0093.
\end{align*}

\subsection{Proof of Lemma~\ref{lem.estimates}(a), (b), and (c)}
\label{sec.a and b}

For (a), Lemma~\ref{lem.degree error} gives
\[
\bigg|\#\{\ell\in V(n):2\nmid\ell\}-\frac{2}{3}v(n)\bigg|
=\big|\operatorname{deg}(2,n)-f(2)v(n)\big|\\
\leq2\bigg(1+\frac{1}{\sqrt{2}}\bigg)\sqrt{n}+1,
\]
which is $\leq (8.5\cdot 10^{-5})\cdot\frac{2}{3}v(n)$ for $n\geq10^{10}$.

For (b), Lemma~\ref{lem.degree error} gives
\[
\bigg|\#\{\ell\in V(n):2\nmid\ell,\,3\nmid\ell\}-\frac{1}{2}v(n)\bigg|
=\big|\operatorname{deg}(6,n)-f(6)v(n)\big|\\
\leq2\bigg(1+\frac{1}{\sqrt{2}}\bigg)\bigg(1+\frac{1}{\sqrt{3}}\bigg)\sqrt{n}+2,
\]
which is $\leq (1.8\cdot 10^{-4})\cdot\frac{1}{2}v(n)$ for $n\geq10^{10}$.

For (c), consider any odd $\ell\leq n$.
Recall that $C(\ell,n)$ is the number of even squarefree numbers $\leq n$ coprime to $\ell$.
This equals the number of odd squarefree numbers $\leq\frac{n}{2}$ coprime to $\ell$, which in turn equals the number of squarefree numbers $\leq\frac{n}{2}$ coprime to $2\ell$, i.e., $\operatorname{deg}(2\ell,\lfloor\frac{n}{2}\rfloor)$.
Also, notice that
\[
v(n)
=2v\bigg(\frac{n}{2}\bigg),
\qquad
f(2\ell)
=\frac{2}{3}f(\ell),
\]
and so
\[
c(\ell,n)
=\frac{1}{3}f(\ell)v(n)
=\frac{2}{3}f(\ell)v\bigg(\frac{n}{2}\bigg)
=f(2\ell)v\bigg(\frac{n}{2}\bigg).
\]
Then Lemmas~\ref{lem.degree error} and~\ref{lem.sigma} together give
\begin{align*}
\big|C(\ell,n)-c(\ell,n)\big|
&=\bigg|\operatorname{deg}\bigg(2\ell,\bigg\lfloor\frac{n}{2}\bigg\rfloor\bigg)-f(2\ell)v\bigg(\frac{n}{2}\bigg)\bigg|\\
&\leq2\sqrt{\frac{n}{2}}\cdot\prod_{p\mid2\ell}\bigg(1+\frac{1}{\sqrt{p}}\bigg)+\frac{1}{2}\prod_{p\mid2\ell}2\\
&=\bigg(1+\frac{1}{\sqrt{2}}\bigg)\sqrt{2n}\cdot\prod_{p\mid\ell}\bigg(1+\frac{1}{\sqrt{p}}\bigg)+\prod_{p\mid\ell}2\\
&\leq\bigg(1+\frac{1}{\sqrt{2}}\bigg)\sqrt{2n}\cdot \alpha_{1/2}\cdot\ell^{\beta_{1/2}}+\alpha_{0}\cdot\ell^{\beta_{0}}\\
&=(\sqrt{2}+1)\alpha_{1/2}\cdot \sqrt{n}\cdot \ell^{\beta_{1/2}}+\alpha_{0}\cdot \ell^{\beta_{0}}
\end{align*}
for all odd $\ell\leq n$.
Meanwhile, Lemma~\ref{lem.sigma} also gives
\[
c(\ell,n)
=\frac{1}{3}f(\ell)v(n)
=\frac{2n/\pi^2}{\prod_{p\mid\ell}(1+\frac{1}{p})}
\geq \frac{2}{\alpha_1\pi^2}\cdot n^{1-\beta_{1}}.
\]
As such, 
\[
\frac{|F(\ell,n)-f(\ell)|}{f(\ell)}
=\frac{|C(\ell,n)-c(\ell,n)|}{c(\ell,n)}
\leq \frac{(\sqrt{2}+1)\alpha_{1/2}\cdot n^{1/2+\beta_{1/2}}+\alpha_{0}\cdot n^{\beta_{0}}}{(2/(\alpha_1\pi^2))\cdot n^{1-\beta_{1}}},
\]
and we numerically verify that this is at most
\[
\frac{9.6390\cdot n^{0.5482}+6.1620\cdot n^{0.2019}}{0.0959\cdot n^{0.9907}} \leq3.8\cdot10^{-3}
\]
when $n\geq10^{10}$.

\subsection{Proof of Lemma~\ref{lem.estimates}(d) and (e)}

For the $3\nmid\ell$ case, we have
\[
\begin{aligned}
\sum_{\substack{\ell\in V(n)\\2\nmid\ell,~3\nmid\ell}}f(\ell)^{-1}
&=\sum_{\substack{\ell\leq n\\\ell\text{ squarefree}\\2\nmid\ell,~3\nmid\ell}}\prod_{p\mid\ell}\bigg(1+\frac{1}{p}\bigg)&&\text{(since $(1-\tfrac{1}{p+1})^{-1}=1+\tfrac{1}{p}$)}\\
&=\sum_{\substack{\ell\leq n\\\ell\text{ squarefree}\\2\nmid \ell,~3\nmid \ell}}\sum_{d\mid \ell}\frac{1}{d}&&\text{($\ell$ is squarefree)}\\
&=\sum_{\substack{d\leq n\\d\text{ squarefree}\\2\nmid d,~3\nmid d}}\frac{1}{d}\sum_{\substack{\ell\leq n\\\ell\text{ squarefree}\\2\nmid \ell,~3\nmid \ell,~d\mid \ell}}1\quad&&\text{(interchange sums)}.
\end{aligned}
\]
We simplify the inner sum by changing variables with $\ell=ad$ and then $b=2a$:
\[
\sum_{\substack{\ell\leq n\\m\text{ squarefree}\\2\nmid\ell,~3\nmid\ell,~d\mid\ell}}1
=\sum_{\substack{a\leq n/d\\a\text{ squarefree}\\2\nmid a,~(a,3d)=1}}1
=\sum_{\substack{b\leq 2\lfloor n/d\rfloor\\b\text{ squarefree}\\2\mid b,~(b,3d)=1}}1
=C\bigg(3d,2\bigg\lfloor\frac{n}{d}\bigg\rfloor\bigg).
\]
In the regime where $n\gg d$, we have 
\[
C\bigg(3d,2\bigg\lfloor\frac{n}{d}\bigg\rfloor\bigg)
\sim c\bigg(3d,2\bigg\lfloor\frac{n}{d}\bigg\rfloor\bigg)
=\frac{1}{3}f(3d)v\bigg(2\bigg\lfloor\frac{n}{d}\bigg\rfloor\bigg)
\sim\frac{4n}{\pi^2}\cdot \frac{f(3d)}{d}.
\]
Furthermore,
\[
\frac{f(3d)}{d^2}
=\frac{1}{d^2}\prod_{p\mid3d}\bigg(1-\frac{1}{p+1}\bigg)
=\bigg(1-\frac{1}{3+1}\bigg)\prod_{p\mid d}\frac{1}{p^2}\bigg(1-\frac{1}{p+1}\bigg)
=\frac{3}{4}\prod_{p\mid d}\frac{1}{p(p+1)}.
\]
This suggests the approximation
\[
\sum_{\substack{d\leq n\\d\text{ squarefree}\\2\nmid d,~3\nmid d}}\frac{1}{d}\cdot\frac{4n}{\pi^2}\cdot \frac{f(3d)}{d}
=\frac{3n}{\pi^2}\sum_{\substack{d\leq n\\d\text{ squarefree}\\2\nmid d,~3\nmid d}}\prod_{p\mid d}\frac{1}{p(p+1)}
=\frac{3n}{\pi^2}\prod_{p\geq5}\bigg(1+\frac{1}{p(p+1)}\bigg)
=\frac{36}{91\zeta(3)}\cdot n.
\]
The error of this approximation is
\begin{align*}
&\bigg|\sum_{\substack{\ell\in V(n)\\2\nmid\ell,~3\nmid\ell}}f(\ell)^{-1}-\frac{36}{91\zeta(3)}\cdot n\bigg|
=\bigg|\sum_{\substack{d\leq n\\d\text{ squarefree}\\2\nmid d,~3\nmid d}}\frac{1}{d}\bigg[
C\bigg(3d,2\bigg\lfloor\frac{n}{d}\bigg\rfloor\bigg)
-\frac{4n}{\pi^2}\cdot \frac{f(3d)}{d}
\bigg]\bigg|\\
&\leq\sum_{\substack{d\leq n\\d\text{ squarefree}\\2\nmid d,~3\nmid d}}\frac{1}{d}\bigg|
C\bigg(3d,2\bigg\lfloor\frac{n}{d}\bigg\rfloor\bigg)
-c\bigg(3d,2\bigg\lfloor\frac{n}{d}\bigg\rfloor\bigg)\bigg|
+\sum_{\substack{d\leq n\\d\text{ squarefree}\\2\nmid d,~3\nmid d}}\frac{1}{d}\bigg|c\bigg(3d,2\bigg\lfloor\frac{n}{d}\bigg\rfloor\bigg)-\frac{4n}{\pi^2}\cdot \frac{f(3d)}{d}
\bigg|\\
&\leq\sum_{\substack{d\leq n\\d\text{ squarefree}\\2\nmid d,~3\nmid d}}\frac{1}{d}\cdot\bigg((\sqrt{2}+1)\alpha_{1/2}\cdot \sqrt{\frac{2n}{d}}\cdot (3d)^{\beta_{1/2}}+\alpha_{0}\cdot(3d)^{\beta_{0}}\bigg)
+\sum_{\substack{d\leq n\\d\text{ squarefree}\\2\nmid d,~3\nmid d}}\frac{1}{d}\cdot\frac{4f(3d)}{\pi^2}\cdot \bigg\{\frac{n}{d}\bigg\}\\
&\leq (2+\sqrt{2})\alpha_{1/2}3^{\beta_{1/2}}\cdot\sqrt{n}\cdot\sum_{d\leq n}d^{\beta_{1/2}-3/2}
+\alpha_{0}3^{\beta_{0}}\cdot\sum_{d\leq n}d^{\beta_{0}-1}
+\frac{4}{\pi^2}\cdot\sum_{d\leq n}d^{-1}\\
&\leq (2+\sqrt{2})\alpha_{1/2}3^{\beta_{1/2}}\zeta(3/2-\beta_{1/2})\cdot\sqrt{n}
+\alpha_{0}3^{\beta_{0}}\cdot\bigg(1+\frac{n^{\beta_{0}}-1}{\beta_{0}}\bigg)
+\frac{4}{\pi^2}\cdot(1+\log n)\\
&\leq40.5553\cdot\sqrt{n}
+7.6917\cdot\bigg(1+\frac{n^{0.2019}-1}{0.2019}\bigg)
+0.4053\cdot(1+\log n).
\end{align*}
For $n\geq10^{10}$, this is $\leq (1.3\cdot10^{-3})\cdot \frac{36}{91\zeta(3)}\cdot n$.

The $3\mid\ell$ case is similar after a change of variables $\ell=3m$:
\[
\sum_{\substack{\ell\in V(n)\\2\nmid\ell,~3\mid\ell}}f(\ell)^{-1}
=\bigg(1+\frac{1}{3}\bigg)\sum_{\substack{m\leq n/3\\m\text{ squarefree}\\2\nmid m,~3\nmid m}}\prod_{p\mid m}\bigg(1+\frac{1}{p}\bigg)
=\frac{4}{3}\sum_{\substack{d\leq n/3\\d\text{ squarefree}\\2\nmid d,~3\nmid d}}\frac{1}{d}\cdot C\bigg(3d,2\bigg\lfloor\frac{n}{3d}\bigg\rfloor\bigg),
\]
and the corresponding error of approximation is
\begin{align*}
&\bigg|\sum_{\substack{\ell\in V(n)\\2\nmid\ell,~3\mid\ell}}f(\ell)^{-1}-\frac{16}{91\zeta(3)}\cdot n\bigg|
=\bigg|\frac{4}{3}\sum_{\substack{d\leq n/3\\d\text{ squarefree}\\2\nmid d,~3\nmid d}}\frac{1}{d}\bigg[
C\bigg(3d,2\bigg\lfloor\frac{n}{3d}\bigg\rfloor\bigg)
-\frac{4n}{3\pi^2}\cdot \frac{f(3d)}{d}
\bigg]\bigg|\\
&\leq\frac{4}{3}\sum_{\substack{d\leq n/3\\d\text{ squarefree}\\2\nmid d,~3\nmid d}}\frac{1}{d}\bigg|
C\bigg(3d,2\bigg\lfloor\frac{n}{3d}\bigg\rfloor\bigg)
-c\bigg(3d,2\bigg\lfloor\frac{n}{3d}\bigg\rfloor\bigg)\bigg|
+\frac{4}{3}\sum_{\substack{d\leq n/3\\d\text{ squarefree}\\2\nmid d,~3\nmid d}}\frac{1}{d}\bigg|c\bigg(3d,2\bigg\lfloor\frac{n}{3d}\bigg\rfloor\bigg)-\frac{4n}{3\pi^2}\cdot \frac{f(3d)}{d}
\bigg|\\
&\leq\frac{4}{3}\sum_{\substack{d\leq n/3\\d\text{ squarefree}\\2\nmid d,~3\nmid d}}\frac{1}{d}\cdot\bigg((\sqrt{2}+1)\alpha_{1/2}\cdot \sqrt{\frac{2n}{3d}}\cdot (3d)^{\beta_{1/2}}+\alpha_{0}\cdot(3d)^{\beta_{0}}\bigg)
+\frac{4}{3}\sum_{\substack{d\leq n/3\\d\text{ squarefree}\\2\nmid d,~3\nmid d}}\frac{1}{d}\cdot\frac{4f(3d)}{\pi^2}\cdot \bigg\{\frac{n}{3d}\bigg\}\\
&\leq\frac{4(2+\sqrt{2})}{3\sqrt{3}}\alpha_{1/2}3^{\beta_{1/2}}\cdot\sqrt{n}\cdot\sum_{d\leq n/3}d^{\beta_{1/2}-3/2}
+\frac{4}{3}\alpha_{0}3^{\beta_{0}}\cdot\sum_{d\leq n/3}d^{\beta_{0}-1}
+\frac{16}{3\pi^2}\cdot\sum_{d\leq n/3}d^{-1}\\
&\leq\frac{4(2+\sqrt{2})}{3\sqrt{3}}\alpha_{1/2}3^{\beta_{1/2}}\zeta(3/2-\beta_{1/2})\cdot\sqrt{n}
+\frac{4}{3}\alpha_{0}3^{\beta_{0}}\cdot\bigg(1+\frac{(n/3)^{\beta_{0}}-1}{\beta_{0}}\bigg)
+\frac{16}{3\pi^2}\cdot\bigg(1+\log \frac{n}{3}\bigg)\\
&\leq31.2195\cdot\sqrt{n}
+10.2556\cdot\bigg(1+\frac{(n/3)^{0.2019}-1}{0.2019}\bigg)
+0.5404\cdot\bigg(1+\log \frac{n}{3}\bigg).
\end{align*}
For $n\geq10^{10}$, this is $\leq (2.2\cdot 10^{-3})\cdot \frac{16}{91\zeta(3)}\cdot n$.

\section{Discussion}
\label{sec.discussion}

In this paper, we showed that the even vertices form a maximum independent set in the squarefree graph by presenting a clique cover of the same size.
A few questions remain.

First, when is the maximum independent set unique?
Notably, it is not unique for $3\le n\le 9$ and $n=21$.
Indeed, for these values of $n$, the multiples of $3$ constitute an independent set of the same size as the multiples of $2$.
(For $n=5$, one can also use the multiples of $5$.)

Second, do the strategies described in Section~\ref{sec.combinatorial} always produce a clique cover of the squarefree graph?
Interestingly, for the greedy strategies, this would produce a single clique cover for all of $\mathbb{N}$.

Finally, considering how Theorem~\ref{thm.main} is implied by Theorem~\ref{thm.stronger}, is Proposition~\ref{prop.chvatal} similarly implied by a corresponding result involving a clique cover?
And does such a result also imply Theorem~\ref{thm.stronger} using an argument similar to Weisenberg's proof of Theorem~\ref{thm.main}?
This would be interesting since \Chvatal{}'s proof argues in terms of intersection families (i.e., independent sets, as opposed to clique covers), while our proof of Theorem~\ref{thm.stronger} makes heavy use of number-theoretic properties of the squarefree graph.
That is, neither approach seems to easily adapt to this potential result.

\section*{Acknowledgments}

DGM was partially supported by NSF DMS 2220304.
WS was partially supported by NSF DMS 2502029 and a Sloan Research Fellowship.

\end{document}